\documentclass[11pt,reqno]{amsart}

\usepackage[T1]{fontenc}
\usepackage{lmodern}
\usepackage{microtype}
\usepackage{mathtools,amssymb,amsmath,amsthm}
\usepackage{geometry}
\usepackage{hyperref}
\usepackage{aliascnt}
\usepackage{array}

\makeatletter
\g@addto@macro\UrlBreaks{\do\_\do\.}
\makeatother

\allowdisplaybreaks[2]
\hypersetup{
  colorlinks=true,
  linkcolor=blue,
  citecolor=blue,
  urlcolor=blue,
  pdfauthor={Yechi Zhou},
  pdftitle={The higher-dimensional Shepp problem: an exact criterion for random ball coverings of tori},
  pdfsubject={Random coverings and the higher-dimensional Shepp problem},
  pdfkeywords={random covering, random ball covering, Shepp criterion, Poisson process, positive association, population extinction}
}

\newcommand{\T}{\mathbb T^d}
\newcommand{\R}{\mathbb R}
\newcommand{\N}{\mathbb N}
\newcommand{\Pp}{\mathbb P}
\newcommand{\Ee}{\mathbb E}
\newcommand{\ind}{\mathbf 1}
\newcommand{\e}{\mathrm e}
\newcommand{\dd}{\,\mathrm d}
\newcommand{\Pois}{\operatorname{Pois}}
\newcommand{\Cov}{\operatorname{Cov}}
\newcommand{\Cloud}{\operatorname{Cloud}}
\newcommand{\Pop}{\operatorname{Pop}}
\newcommand{\whatN}{\widehat N}
\newcommand{\whatA}{\widehat A}

\newtheorem{theorem}{Theorem}[section]
\newaliascnt{lemma}{theorem}
\newtheorem{lemma}[lemma]{Lemma}
\aliascntresetthe{lemma}
\newaliascnt{proposition}{theorem}
\newtheorem{proposition}[proposition]{Proposition}
\aliascntresetthe{proposition}

\usepackage[nameinlink,capitalize,noabbrev]{cleveref}
\crefname{theorem}{Theorem}{Theorems}
\crefname{lemma}{Lemma}{Lemmas}
\crefname{proposition}{Proposition}{Propositions}

\title[The higher-dimensional Shepp problem]{The higher-dimensional Shepp problem: an exact criterion for random ball coverings of tori}
\author{Yechi Zhou}
\address{School of Mathematical Sciences, Fudan University, Shanghai, China}
\email{25210180114@m.fudan.edu.cn}
\date{}

\begin{document}

\begin{abstract}
We solve the Euclidean-ball case of the higher-dimensional Shepp covering
problem.  More precisely, we give an exact criterion for full limsup coverage
of the $d$-dimensional torus, $d\ge2$, by independently centered Euclidean
balls with an arbitrary decreasing sequence of radii.  Let
$X_1,X_2,\ldots$ be independent Haar-uniform points, let
$r_1\ge r_2\ge\cdots\downarrow0$, and put
$u_n(z)=m(B(0,r_n)\cap B(z,r_n))$ and $H(z)=\sum_{n\ge1}u_n(z)$.  Then every
point belongs to infinitely many of the balls $B(X_n,r_n)$ almost surely if
and only if $\int_{\mathbb T^d}\exp(H)\,\dd m=\infty$.  In dimension one this
condition is equivalent to Shepp's criterion.  No regular-variation or
comparable-scale assumption is imposed on the radii.  The main difficulty is
shared noise: after spatial decomposition, the same Poisson input acts on many
uncovered cells, so their descendants are not conditionally independent.  We
overcome this by establishing an extinction bound for monotone population
recursions driven by positively associated innovations.  Together with
Poissonization and spatial localization this proves sufficiency, while a
second-moment estimate and the zero--one law prove necessity.
\end{abstract}

\maketitle

\section{Introduction}

Given an arbitrary decreasing sequence of radii, when do independently
centered balls cover every point of a higher-dimensional torus infinitely
often?  This is the ball version of a classical random covering problem.
Dvoretzky introduced the circle problem \cite{Dvoretzky1956}, and Shepp found
its exact criterion \cite{Shepp1972}; see also Kahane \cite{Kahane1985} and
subsequent work on small random intervals such as \cite{FanWu2004}.

The higher-dimensional problem has resisted the one-dimensional theory.  Work
of Hoffmann-J\o rgensen \cite{HoffmannJorgensen1973} and El H\'elou
\cite{ElHelou1978} gave general results for metric spaces and tori, while
Kahane obtained a complete solution for homothetic simplexes and explicitly
left the corresponding ball and cube problems open \cite{Kahane1990}.
Kahane's higher-dimensional sufficient condition does not apply to Euclidean
balls, as Bierm\'e and Estrade later emphasized \cite{BiermeEstrade2012}.
Related developments include Janson's sharp asymptotic laws for small
translated sets \cite{Janson1986} and the study of random limsup sets through
Hausdorff dimension and large-intersection properties
\cite{JarvenpaaEtAl2014,Persson2015}.

We give an exact criterion for the Euclidean-ball problem in every dimension
$d\ge2$ and for every prescribed nonincreasing radius sequence.  The condition
is expressed directly through the overlaps of the balls.

Let $m$ denote normalized Haar measure on $\mathbb T^d$, let
$X_1,X_2,\ldots$ be independent with law $m$, and let
$r_1\ge r_2\ge\cdots\downarrow0$.  Write
$E_r=\limsup_n B(X_n,r_n)$ for the set of points covered infinitely often.
For $z\in\mathbb T^d$, define the overlap function at scale $n$ by the volume
of two radius-$r_n$ balls whose centers differ by $z$.  The main result is
\begin{equation}\label{eq:introduction-criterion}
 \left.
 \begin{aligned}
  u_n(z)&:=m\bigl(B(0,r_n)\cap B(z,r_n)\bigr),\\
  H(z)&:=\sum_{n\ge1}u_n(z)
 \end{aligned}
 \right\}
 \qquad
 \Pp(E_r=\mathbb T^d)=1
 \quad\Longleftrightarrow\quad
 \int_{\mathbb T^d}\e^{H(z)}\,\dd m(z)=\infty.
\end{equation}
The criterion has two features worth emphasizing.  It is intrinsic: it is
expressed only through the overlaps of the prescribed balls and not through
the auxiliary scales used in the proof.  It also applies to every decreasing
radius sequence, without regular-variation or comparable-scale assumptions.
When $d=1$, the overlap integral is equivalent to Shepp's series criterion;
this reduction is given in \Cref{sec:model}.  Thus
\eqref{eq:introduction-criterion} supplies the higher-dimensional ball
counterpart of the exact circle theorem.

The sufficiency direction contains the main difficulty.  After Poissonization
and spatial decomposition, the uncovered set is represented by the dyadic
cells that it intersects.  This population is not a branching process: the
same newly revealed Poisson configuration acts on every current cell, so the
descendants of different cells are not conditionally independent.  The
available structure is instead monotonicity of the deletion maps together
with positive association of increasing functions of a Poisson process
\cite[Theorem~20.4]{LastPenrose2018}.

To handle this shared noise, we establish an extinction criterion for
positively associated population recursions.  The abstract recursion is
defined in \Cref{sec:extinction}: $\Xi_j$ is the population at time $j$,
$m_j=\Ee|\Xi_j|$ is its actual mean, and $\delta_j$ is the uniform
one-particle killing lower bound; see
\eqref{eq:shared-update}, \eqref{eq:actual-mean}, and
\eqref{eq:abstract-killing}.  A backward survival function and a quadratic
potential give
\begin{equation}\label{eq:introduction-survival-bound}
 \Pp(\Xi_N\ne0)
 \le
 \left[
  \frac{2}{3m_N}
  +\frac16\sum_{j<N}\frac{\delta_j}{m_j}
 \right]^{-1}.
\end{equation}
Consequently, divergence of $\sum_j\delta_j/m_j$ forces extinction.  In the
geometric application, direct estimates of the actual uncovered-cell process
show that the integral condition in \eqref{eq:introduction-criterion} implies
this divergence.  The resulting bound is a shared-noise analogue of classical
survival estimates for branching processes in varying environments; the
precise comparison is given in \Cref{sec:extinction}.

For necessity, the normalized uncovered volume has a uniformly bounded
second-moment ratio whenever $\int\e^H\,\dd m<\infty$.  The
Paley--Zygmund inequality then gives a positive probability of a residual
uncovered point, and Kolmogorov's zero--one law upgrades this to almost-sure
failure of full limsup covering.

The paper is organized as follows.  \Cref{sec:model} states the model and the
main theorem.  The next two sections convert the exponential-overlap condition
into quantitative information at a sequence of spatial scales and construct
the corresponding Poisson decomposition.  The population estimate is then
proved abstractly and applied to the uncovered dyadic cells.  The final three
sections remove the Poissonization, prove necessity, and assemble the two
directions.  Measurability and the standard probabilistic inputs are recorded
in the appendices.

\section{Model and main result}\label{sec:model}

Let
\[
 \T:=(\R/\mathbb Z)^d,\quad d\ge2,\quad
 X_n\stackrel{\mathrm{iid}}{\sim}m,\qquad
 r_1\ge r_2\ge\cdots>0,\quad r_n\downarrow0,
\]
where $d_{\T}$ is the quotient Euclidean metric, $m$ is normalized Haar
measure, and, for every metric space $(E,d_E)$,
\[
 m(\T)=1,\qquad |z|:=d_{\T}(z,0),\qquad
 B_E(x,r):=\{y\in E:d_E(x,y)<r\},\qquad
 \kappa_d:=\mathcal L^d(B_{\R^d}(0,1)).
\]
We abbreviate $B_{\T}$ to $B$ and set
\[
 E_r:=\bigcap_{M\ge1}\bigcup_{n\ge M}B(X_n,r_n),\qquad
 u_n(z):=m(B(0,r_n)\cap B(z,r_n)),\qquad
 H(z):=\sum_{n\ge1}u_n(z).
\]

\begin{theorem}[Main criterion]\label{thm:main}
\[
 \Pp(E_r=\T)=1
 \quad\Longleftrightarrow\quad
 \int_{\T}\e^{H(z)}\,\dd m(z)=\infty.
\]
\end{theorem}

\paragraph{One-dimensional normalization.}
On the circle, write \(\lambda_n:=2r_n\) for the arc length and delete a
finite prefix so that \(\lambda_n<1/2\).  In the coordinate
\(t\in[-1/2,1/2]\),
\[
 u_n(t)=(\lambda_n-|t|)_+.
\]
If \(L_N:=\sum_{n\le N}\lambda_n\) and
\(\lambda_{N+1}\le t<\lambda_N\), then
\[
 \sum_{n\ge1}(\lambda_n-t)_+=L_N-Nt,
\]
and therefore
\begin{equation}\label{eq:d1-check}
 \int_0^{\lambda_1}
 \exp\!\left(\sum_{n\ge1}(\lambda_n-t)_+\right)\,\dd t
 =
 \sum_{N\ge1}\e^{L_N}
 \int_{\lambda_{N+1}}^{\lambda_N}\e^{-Nt}\,\dd t.
\end{equation}
The standard summation-by-parts comparison for decreasing
\((\lambda_n)_n\) gives
\[
 \int_0^{\lambda_1}
 \exp\!\left(\sum_{n\ge1}(\lambda_n-t)_+\right)\,\dd t=\infty
 \quad\Longleftrightarrow\quad
 \sum_{N\ge1}\frac{\e^{L_N}}{N^2}=\infty,
\]
which is Shepp's criterion \cite{Shepp1972,Kahane1985}.

For $M\ge0$,
\[
 H^{(M)}:=\sum_{n>M}u_n,\qquad
 0\le H-H^{(M)}\le\sum_{n\le M}v_n\le M,\qquad
 v_n:=m(B(0,r_n)).
\]
Hence
\[
 \int\e^H\,\dd m=\infty
 \Longleftrightarrow
 \int\e^{H^{(M)}}\,\dd m=\infty,\qquad
 \limsup_nB(X_n,r_n)=\limsup_{n>M}B(X_n,r_n).
\]
After deleting a finite prefix,
\[
 r_n<\frac14,\qquad v_n=\kappa_dr_n^d,\qquad
 A_0:=S_0:=0,\qquad
 A_N:=\sum_{n\le N}v_n,\qquad
 S_N:=\sum_{n\le N}\frac{v_n}{r_n}.
\]
Throughout,
\[
 0<c_d,C_d<\infty,\qquad c_d,C_d\text{ depend only on }d
\]
and may change from line to line.
Translation invariance gives
\begin{equation}\label{eq:int-u}
 \int_{\T}u_n(z)\,\dd m(z)=v_n^2.
\end{equation}

The sufficiency dependency chain is
\begin{equation}\label{eq:dependency-chain}
\begin{aligned}
 \int_{\T}\e^H\,\dd m=\infty
 &\overset{\text{\Cref{lem:analytic-mass}}}{\Longrightarrow}
 \text{localized scale-mass divergence}\\
 &\overset{\text{\Cref{eq:actual-resistance-diverges}}}{\Longrightarrow}
 \text{shared-noise resistance divergence}\\
 &\overset{\text{\Cref{thm:shared-resistance}}}{\Longrightarrow}
 \text{finite-time extinction of the residual set}\\
 &\overset{\text{\Cref{thm:poisson-sufficiency}}}{\Longrightarrow}
 \text{Poisson tail covering}
 \overset{\text{\Cref{thm:sufficiency}}}{\Longrightarrow}
 \Pp(E_r=\T)=1.
\end{aligned}
\end{equation}

\section{Analytic localization}\label{sec:analytic-localization}

This section converts the intrinsic condition
$\int_{\T}\e^H\,\dd m=\infty$ into a discrete sum indexed by spatial
scales.  We first compare the exact overlap profile with two cones, then
localize all possible divergence at the origin, and finally group the radii
according to the smaller of their geometric and cumulative-intensity scales.

\begin{lemma}[Overlap profile]\label{lem:overlap-profile}
For $0\le t\le2r$, define
\[
 u_r(t):=\mathcal L^d(B_{\R^d}(0,r)\cap B_{\R^d}(te_1,r))
       =\kappa_dr^d\psi_d(t/r).
\]
For $0<r<1/4$ and $z\in\T$,
\[
 m(B(0,r)\cap B(z,r))
 =\begin{cases}
   u_r(|z|),&|z|\le2r,\\
   0,&|z|>2r.
  \end{cases}
\]
Moreover,
\begin{align}
 \psi_d(x)
 &=\frac{2\kappa_{d-1}}{\kappa_d}
   \int_{x/2}^{1}(1-s^2)^{(d-1)/2}\,\dd s,\label{eq:psi}\\
 \psi_d'(x)
 &=-\frac{\kappa_{d-1}}{\kappa_d}
   (1-x^2/4)^{(d-1)/2},\label{eq:psi-prime}\\
 \left(1-\frac{\kappa_{d-1}}{\kappa_d}x\right)_+
 &\le\psi_d(x)\le(1-x/2)_+.
 \label{eq:cone-bounds}
\end{align}
Here $\psi_d(x):=0$ for $x\ge2$.
\end{lemma}

\begin{proof}
Slicing the Euclidean lens perpendicular to $e_1$ gives
\[
 u_r(t)
 =2\kappa_{d-1}\int_{t/2}^{r}
 (r^2-s^2)^{(d-1)/2}\,\dd s.
\]
The change of variables $s=ry$ proves \eqref{eq:psi}, and differentiation
proves \eqref{eq:psi-prime}.  In particular,
\[
 \psi_d(0)=1,\qquad \psi_d(2)=0,
 \qquad
 \psi_d''\ge0,\qquad
 \operatorname{Tan}_{0}\psi_d\le\psi_d
 \le\operatorname{Chord}_{(0,1),(2,0)}\psi_d.
\]
The tangent and chord are respectively
$1-(\kappa_{d-1}/\kappa_d)x$ and $1-x/2$; nonnegativity of $\psi_d$
then gives \eqref{eq:cone-bounds}.  If $r<1/4$, both balls lift uniquely to
Euclidean balls whenever they intersect, which proves the asserted torus
formula.
\end{proof}

Define
\begin{equation}\label{eq:G}
 G:(0,\infty)\to[0,\infty],\qquad
 G(t):=\sum_{n\ge1}v_n(1-t/r_n)_+.
\end{equation}
Consequently, there exists $t_{\mathrm{loc}}>0$ such that
\[
 0<|z|\le t_{\mathrm{loc}}
 \Longrightarrow
 G\!\left(\frac{\kappa_{d-1}}{\kappa_d}|z|\right)
 \le H(z)\le G(|z|/2).
\]
For every $c>0$,
\[
 \int_0^\varepsilon\e^{G(ct)}t^{d-1}\,\dd t
 =c^{-d}\int_0^{c\varepsilon}\e^{G(s)}s^{d-1}\,\dd s.
\]
For each $t>0$, only finitely many summands in $G(t)$ are nonzero because
$r_n\downarrow0$.  Hence changing a positive upper integration limit cannot
create or remove divergence; any divergence occurs at $t=0$.
Fix $t_0\in(0,\min\{t_{\mathrm{loc}},1/4\}]$.  Choose $N_0$ so that
$2r_n<t_0$ for $n>N_0$.  If $|z|\ge t_0$, then $u_n(z)=0$ for
$n>N_0$, and hence
\[
 \sup_{\{|z|\ge t_0\}}H(z)
 \le\sum_{n\le N_0}v_n<\infty.
\]
On $B(0,t_0)$, Haar measure agrees with Euclidean volume, so radial
integration gives
\[
 \int_{B(0,t_0)}f(|z|)\,\dd m(z)
 =d\kappa_d\int_0^{t_0}f(t)t^{d-1}\,\dd t
\]
for every nonnegative Borel $f$.  The two cone comparisons and the preceding
change of variables therefore show that
\begin{equation}\label{eq:cone-equivalence}
 \int_{\T}\e^H\,\dd m=\infty
 \quad\Longleftrightarrow\quad
 \int_0^{t_0}\e^{G(t)}t^{d-1}\,\dd t=\infty.
\end{equation}

For $r_{N+1}<t<r_N$,
\[
 G(t)=\sum_{n\le N}v_n(1-t/r_n)=A_N-S_Nt.
\]
Consequently,
\begin{equation}\label{eq:explicit-series}
 \int_{\T}\e^H\,\dd m=\infty
 \Longleftrightarrow
 \sum_{N\ge1}\e^{A_N}
 \int_{r_{N+1}}^{r_N}\e^{-S_Nt}t^{d-1}\,\dd t=\infty.
\end{equation}
Indeed, \eqref{eq:G} equals $A_N-S_Nt$ on
$(r_{N+1},r_N)$; the finitely many intervals not contained in
$(0,t_0)$ have finite total contribution and do not affect divergence.

Put
\[
 \sigma_N:=\min\{r_N,d/S_N\},\qquad \ell_k:=2^{-k},
\]
with $d/0:=\infty$.  Choose $k_*$ by
\[
 \ell_{k_*}\le\sigma_1<2\ell_{k_*},
\]
and set
\begin{equation}\label{eq:packets}
 \begin{aligned}
  \whatN_{k_*-1}&:=0,
  &\whatA_{k_*-1}&:=A_0=0,\\
 \whatN_k&:=\max\{N:\sigma_N\ge\ell_k\},
 &\whatA_k&:=A_{\whatN_k},
 \qquad k\ge k_*.
\end{aligned}
\end{equation}
Thus $\sigma_N$ is the effective spatial scale after the first $N$ radius
types, $\whatN_k$ is the last type visible at scale $\ell_k$, and $\whatA_k$
is the cumulative volume mass through that type.
Then
\begin{equation}\label{eq:packet-interface}
 \whatN_k\uparrow\infty,\qquad
 \ell_kS_{\whatN_k}\le d,\qquad
 n\le\whatN_k\Longrightarrow r_n\ge\ell_k,\qquad
 \whatN_{k-1}<n\le\whatN_k\Longrightarrow\sigma_n<2\ell_k.
\end{equation}
To verify these properties, note first that $(\sigma_N)_N$ is nonincreasing
and tends to zero: both $r_N$ and $d/S_N$ are nonincreasing, and
$r_N\downarrow0$.  Thus the maximum in \eqref{eq:packets} exists and is
finite, while $\widehat N_k\uparrow\infty$ because $\ell_k\downarrow0$.
If $N\le\widehat N_k$, then
$\sigma_N\ge\ell_k$; since
$\sigma_N=\min\{r_N,d/S_N\}$, this implies both
$r_N\ge\ell_k$ and $\ell_kS_N\le d$.  For the last assertion, if
$k=k_*$, then $\sigma_n\le\sigma_1<2\ell_{k_*}$; if $k>k_*$, then
$n>\widehat N_{k-1}$ and the definition gives
$\sigma_n<\ell_{k-1}=2\ell_k$.

\begin{lemma}[Packet domination]\label{lem:packet}
For $t>0$,
\begin{equation}\label{eq:packet-domination}
 \e^{G(t)}
 \le1+\e^d\sum_{k\ge k_*}
 (\e^{\whatA_k}-\e^{\whatA_{k-1}})
 \exp\!\left(-\frac{dt}{2\ell_k}\right).
\end{equation}
\end{lemma}

\begin{proof}
Put
\[
 N(t):=\max\bigl(\{n:r_n>t\}\cup\{0\}\bigr).
\]
Then
\[
 G(t)=A_{N(t)}-tS_{N(t)},\qquad
 \frac{dt}{\sigma_n}\le d+tS_{N(t)}\quad(n\le N(t)).
\]
Thus
\begin{align*}
 1+\e^d\sum_{n\ge1}(\e^{A_n}-\e^{A_{n-1}})\e^{-dt/\sigma_n}
 &\ge1+\e^{-tS_{N(t)}}(\e^{A_{N(t)}}-1)\\
 &\ge\e^{A_{N(t)}-tS_{N(t)}}=\e^{G(t)}.
\end{align*}
Finally,
\[
 \e^{-dt/\sigma_n}\le\e^{-dt/(2\ell_k)},\qquad
 \sum_{\whatN_{k-1}<n\le\whatN_k}(\e^{A_n}-\e^{A_{n-1}})
 =\e^{\whatA_k}-\e^{\whatA_{k-1}}.
\]
\end{proof}

\begin{lemma}[Divergence of the scale sum]\label{lem:analytic-mass}
\[
 \int_{\T}\e^H\,\dd m=\infty
 \Longrightarrow
 \sum_{k\ge k_*}\ell_k^d
 (\e^{\whatA_k}-\e^{\whatA_{k-1}})=\infty.
\]
\end{lemma}

\begin{proof}
By \eqref{eq:cone-equivalence} and \eqref{eq:packet-domination},
\begin{align*}
 \infty
 &=\int_0^{t_0}\e^{G(t)}t^{d-1}\,\dd t\\
 &\le \frac{t_0^d}{d}
 +\e^d\sum_k(\e^{\whatA_k}-\e^{\whatA_{k-1}})
 \int_0^\infty\e^{-dt/(2\ell_k)}t^{d-1}\,\dd t\\
 &=\frac{t_0^d}{d}+C_d\sum_k\ell_k^d
 (\e^{\whatA_k}-\e^{\whatA_{k-1}}).
\end{align*}
\end{proof}

\section{Poisson decomposition and dyadic refinement}
\label{sec:poisson-decomposition}

We now construct the geometric population to which the extinction criterion
will be applied.  At scale $\ell_k$, an individual is a dyadic cell together
with the portion of the uncovered set lying in that cell.  The radii assigned
to this scale are revealed in independent Poisson blocks of total volume
intensity at most one; after the last block, every surviving cell is refined
into its $2^d$ children.
Throughout \Cref{sec:poisson-decomposition,sec:spatial-application}, $n$
indexes radius types, $k$ spatial scales, $p$ Poisson blocks, $\alpha$ dyadic
cells, and $\varepsilon$ children; \Cref{sec:spatial-application} embeds
$(k,p)$ into the abstract time $j$.

Choose $\eta\in\{2^{-s}:s\in\mathbb N_0\}$ so that
\[
 h_k:=\eta\ell_k,\qquad
 \rho_k:=\frac{\sqrt d}{2}h_k\le\frac14\ell_k.
\]
Define the label set and cell map by
\begin{align}
 \mathcal I_k
 &:=\{0,1,\ldots,h_k^{-1}-1\}^d,\label{eq:label-set}\\
 Q_k:\mathcal I_k&\longrightarrow\mathcal P(\T),\qquad
 \alpha\longmapsto Q_{k,\alpha}\notag\\
 Q_{k,\alpha}
 &:=\{x+\mathbb Z^d:x\in\R^d,\ 
       \alpha_ih_k\le x_i\le(\alpha_i+1)h_k,\ 1\le i\le d\}.
 \label{eq:cell-map}
\end{align}
For later use, set
\[
 c_{k,\alpha}
 :=\bigl((\alpha_i+\tfrac12)h_k\bigr)_{i=1}^d+\mathbb Z^d
 \in Q_{k,\alpha}.
\]
Thus
\[
 |\mathcal I_k|=h_k^{-d}=\eta^{-d}\ell_k^{-d}.
\]
Since $h_{k+1}=h_k/2$, define
\begin{equation}\label{eq:child-map}
 \operatorname{ch}_k:
 \mathcal I_k\times\{0,1\}^d\longrightarrow\mathcal I_{k+1},
 \qquad
 \operatorname{ch}_k(\alpha,\varepsilon)
 :=(2\alpha_i+\varepsilon_i)_{i=1}^d.
\end{equation}
Then
\begin{equation}\label{eq:child-properties}
 Q_{k,\alpha}
 =\bigcup_{\varepsilon\in\{0,1\}^d}
 Q_{k+1,\operatorname{ch}_k(\alpha,\varepsilon)},
 \qquad
 Q_{k+1,\operatorname{ch}_k(\alpha,\varepsilon)}
 \subseteq Q_{k,\alpha}.
\end{equation}

Put
\[
 \mathcal N_k:=\{n:\whatN_{k-1}<n\le\whatN_k\},\qquad
 \Delta A_k:=\whatA_k-\whatA_{k-1},\qquad
 M_k:=\lceil\Delta A_k\rceil.
\]
The intervals $(\mathcal L_{k,n})_{n\in\mathcal N_k}$ concatenate the
individual masses $v_n$ into $[0,\Delta A_k)$, while
$(\mathcal B_{k,p})_{1\le p\le M_k}$ cut that interval into pieces of length
at most one.  Their intersections specify what fraction of the Poisson
intensity of radius type $n$ is assigned to block $p$.
For $n\in\mathcal N_k$, $1\le p\le M_k$, set
\[
\begin{array}{@{}c@{\qquad}c@{}}
\displaystyle
 \mathcal L_{k,n}
 :=\left[\sum_{\substack{m\in\mathcal N_k\\m<n}}v_m,
          \sum_{\substack{m\in\mathcal N_k\\m\le n}}v_m\right)
&
\displaystyle
 \mathcal B_{k,p}
 :=\bigl[\min\{p-1,\Delta A_k\},\min\{p,\Delta A_k\}\bigr)
\\[12pt]
\displaystyle
 \theta_{n,p}
 :=\frac{\mathcal L^1(\mathcal L_{k,n}\cap\mathcal B_{k,p})}{v_n}
&
\displaystyle
 b_{k,p}
 :=\sum_{n\in\mathcal N_k}\theta_{n,p}v_n
  =\mathcal L^1(\mathcal B_{k,p}).
\end{array}
\]
Hence
\begin{equation}\label{eq:block-identities}
 0<b_{k,p}\le1,\qquad
 \sum_{p=1}^{M_k}\theta_{n,p}=1\quad(n\in\mathcal N_k),\qquad
 \sum_{p=1}^{M_k}b_{k,p}=\Delta A_k.
\end{equation}
Indeed, the $\mathcal B_{k,p}$ partition $[0,\Delta A_k)$ and the
$\mathcal L_{k,n}$ do the same.  Summing intersection lengths first in $p$
or first in $n$ gives the two identities in \eqref{eq:block-identities}; the
first follows from the definition of $M_k$.

For every standard Borel space $E$, put
\[
 \Cloud(E):=\mathbb N_0\times E^{\mathbb N},
 \qquad
 \omega=(q,\mathbf x),
 \qquad
 [\omega]:=\{\mathbf x(i):0\le i<q\}.
\]
Only the active prefix $0\le i<q$ is used.  The inactive tail is a common
i.i.d. coordinate reservoir for measurable coupling and superposition and has
no geometric role.  If $F$ is finite, set
\[
 \Cloud_F(E):=\prod_{a\in F}\Cloud(E),
 \qquad
 \omega_a=(q_a,\mathbf x_a).
\]
For a metric space $(E,d_E)$ and $\rho:F\to(0,\infty)$, define
\begin{equation}\label{eq:enumerated-cover}
 \Cov_{E,\rho}:\Cloud_F(E)\longrightarrow\mathcal P(E),
 \qquad
 \Cov_{E,\rho}(\omega)
 :=\bigcup_{a\in F}\ \bigcup_{0\le i<q_a}
 B_E(\mathbf x_a(i),\rho(a)).
\end{equation}
For a single cloud and radius $s>0$, write
$\Cov_{E,s}(q,\mathbf x):=\bigcup_{0\le i<q}B_E(\mathbf x(i),s)$.
Finite clouds are concatenated by
\[
 (q,\mathbf x)\oplus(q',\mathbf x')
 :=(q+q',\mathbf z),
 \qquad
 \mathbf z(i):=
 \begin{cases}
  \mathbf x(i),&i<q,\\
  \mathbf x'(i-q),&i\ge q.
 \end{cases}
\]

On a product probability space, take mutually independent
\begin{equation}\label{eq:fractional-subprocess}
 \Pi_{k,p}^{(n)}
 =\bigl(P_{k,p}^{(n)},X_{k,p}^{(n)}\bigr)\in\Cloud(\T),
 \quad
 \substack{k\ge k_*,\ n\in\mathcal N_k,\\1\le p\le M_k},
\end{equation}
with
\[
 P_{k,p}^{(n)}\sim\Pois(\theta_{n,p}),
 \qquad
 \bigl(X_{k,p}^{(n)}(i)\bigr)_{i\ge0}
 \stackrel{\mathrm{iid}}{\sim}m,
 \qquad
 P_{k,p}^{(n)}\perp\!\!\!\perp X_{k,p}^{(n)}.
\]
Define
\begin{equation}\label{eq:block-and-full-processes}
 \boldsymbol\Pi_{k,p}
 :=\bigl(\Pi_{k,p}^{(n)}\bigr)_{n\in\mathcal N_k}
 \in\Cloud_{\mathcal N_k}(\T),
 \qquad
 \Pi_n^{\mathrm{full}}
 :=\bigoplus_{p=1}^{M_k}\Pi_{k,p}^{(n)}
 \quad(n\in\mathcal N_k).
\end{equation}
Poisson superposition yields
\begin{equation}\label{eq:poisson-superposition}
 \begin{aligned}
 &(\boldsymbol\Pi_{k,p})_{\substack{k\ge k_*\\1\le p\le M_k}}
   \text{ are mutually independent},\\
 &(\Pi_n^{\mathrm{full}})_1\sim\Pois(1),
 \quad
 (\Pi_n^{\mathrm{full}})_1
 \perp\!\!\!\perp
 \bigl((\Pi_n^{\mathrm{full}})_2(i)\bigr)_{i\ge0},\\
 &
 \bigl((\Pi_n^{\mathrm{full}})_2(i)\bigr)_{i\ge0}
 \stackrel{\mathrm{iid}}{\sim}m,
 \quad
 (\Pi_n^{\mathrm{full}})_{n\ge1}\text{ are independent},\\
 &\sum_{n\in\mathcal N_k}\theta_{n,p}v_n=b_{k,p}
   \quad(1\le p\le M_k).
 \end{aligned}
\end{equation}
For the rest of the proof, abbreviate
\[
 \Cov(\boldsymbol\Pi_{k,p})
 :=\Cov_{\T,r|_{\mathcal N_k}}(\boldsymbol\Pi_{k,p}),
 \qquad
 \Cov(\Pi_n^{\mathrm{full}})
 :=\Cov_{\T,r_n}(\Pi_n^{\mathrm{full}}).
\]
For $k\ge k_*$ and $0\le p\le M_k$, define
\begin{align}
 a_{k,p}
 &:=\whatA_{k-1}+\sum_{q=1}^{p}b_{k,q}
    =\whatA_{k-1}+\min\{p,\Delta A_k\},\label{eq:akp}\\
 \mathfrak R_{k,p}
 &:=\T\setminus\left(
    \bigcup_{n=1}^{\whatN_{k-1}}\Cov(\Pi_n^{\mathrm{full}})
    \cup\bigcup_{q=1}^{p}\Cov(\boldsymbol\Pi_{k,q})\right),\label{eq:Rkp}\\
 m_{k,p}
 &:=\Ee\#\{\alpha\in\mathcal I_k:
       \mathfrak R_{k,p}\cap Q_{k,\alpha}\ne\varnothing\}.
 \label{eq:mean-count}
\end{align}
where $\sum_{q=1}^{0}b_{k,q}:=0$ and
$\bigcup_{q=1}^{0}\Cov(\boldsymbol\Pi_{k,q}):=\varnothing$.  In particular,
\[
 \mathfrak R_{k,p}\subseteq\T\text{ is compact},\qquad
 0\le m_{k,p}\le|\mathcal I_k|.
\]

\begin{lemma}[One-cell estimates]\label{lem:one-cell}
For $1\le p\le M_k$,
\begin{align}
 m_{k,p-1}
 &\le C_d\ell_k^{-d}\e^{-a_{k,p-1}},\label{eq:mean-bound}\\
 \inf_{\substack{\alpha\in\mathcal I_k\\
                  \varnothing\ne A\subseteq Q_{k,\alpha}\text{ compact}}}
 \Pp(A\setminus\Cov(\boldsymbol\Pi_{k,p})=\varnothing)
 &\ge1-\e^{-q_db_{k,p}}
 \ge(1-\e^{-q_d})b_{k,p},
 \qquad q_d:=(3/4)^d.
 \label{eq:killing-bound}
\end{align}
\end{lemma}

\begin{proof}
For the cell center $c_{k,\alpha}$ introduced immediately after
\eqref{eq:cell-map}, put
\[
 \Theta_n^{k,p-1}:=
 \begin{cases}
  1,&n\le\whatN_{k-1},\\
  \sum_{q=1}^{p-1}\theta_{n,q},&n\in\mathcal N_k,\\
  0,&n>\whatN_k.
 \end{cases}
 \qquad
 \sum_n\Theta_n^{k,p-1}v_n=a_{k,p-1}.
\]
Every point of $Q_{k,\alpha}$ lies within distance $\rho_k$ of
$c_{k,\alpha}$.  Hence, for every $n\le\widehat N_k$,
\[
 x\in B(c_{k,\alpha},r_n-\rho_k)
 \Longrightarrow Q_{k,\alpha}\subseteq B(x,r_n),
 \qquad r_n\ge\ell_k.
\]
If the residual set meets $Q_{k,\alpha}$, none of the previously revealed
balls can have its center in one of these smaller concentric balls.  The
relevant counts are independent Poisson variables, and therefore
\begin{align*}
 \Pp(\mathfrak R_{k,p-1}\cap Q_{k,\alpha}\ne\varnothing)
 &\le\exp\!\left[-\sum_n\Theta_n^{k,p-1}
          \kappa_d(r_n-\rho_k)^d\right],\\
 0\le v_n-\kappa_d(r_n-\rho_k)^d
 &\le C_d\rho_k\frac{v_n}{r_n},\\
 \sum_n\Theta_n^{k,p-1}\kappa_d(r_n-\rho_k)^d
 &\ge a_{k,p-1}
   -C_d\ell_k\sum_{n\le\whatN_k}\frac{v_n}{r_n}
 \ge a_{k,p-1}-C_d.
\end{align*}
The second line is the mean-value estimate
$r_n^d-(r_n-\rho_k)^d\le C_d\rho_kr_n^{d-1}$; the last line uses
$\rho_k\le\ell_k/4$ and
$\ell_kS_{\widehat N_k}\le d$ from
\eqref{eq:packet-interface}.
Summation over $|\mathcal I_k|=\eta^{-d}\ell_k^{-d}$ labels gives
\eqref{eq:mean-bound}.

For the new block, the total number of balls that contain the whole cell is
Poisson with mean
\[
 \sum_{n\in\mathcal N_k}\theta_{n,p}
 \kappa_d(r_n-\rho_k)^d.
\]
Moreover,
\[
 \kappa_d(r_n-\rho_k)^d\ge(3/4)^d\kappa_dr_n^d=q_dv_n,
\]
so this mean is at least $q_db_{k,p}$.  The presence of one such ball deletes
every nonempty compact $A\subseteq Q_{k,\alpha}$.  Consequently the deletion
probability is at least $1-\e^{-q_db_{k,p}}$.  Finally, concavity on
$[0,1]$ gives
\[
 1-\e^{-q_db}\ge(1-\e^{-q_d})b,\qquad0\le b\le1.\qedhere
\]
\end{proof}

\section{An extinction criterion for positively associated population recursions}
\label{sec:extinction}

The population considered here is not a branching process: at time $j$, the
same random input $W_j$ acts on every current particle.  Conditional
independence between different families is therefore unavailable.  The
replacement is an order structure in which a larger input is more destructive
and the law of $W_j$ is positively associated.  A backward one-particle
survival function will then dominate the survival probability of the full
population.

Classical survival bounds for branching processes in varying environments
use independent reproduction between individuals
\cite{Agresti1975,Kersting2020}.  The same conditional independence remains
present when the environment is positively correlated across generations
\cite{ChenGuillotinPlantard2019}.  The recursion below replaces that branching
property by destructive monotonicity and positive association of the shared
innovation.

For \(j\ge0\), let \((\mathcal S_j^\dagger,\preceq_j)\) be a standard Borel
partially ordered space with
\[
 \mathcal S_j^\dagger=\mathcal S_j\sqcup\{\dagger\},
 \qquad
 \dagger\preceq_jx\quad(x\in\mathcal S_j^\dagger).
\]
Let \(\mathcal C_j\) be a finite set equipped with a fixed total order, let
\((\mathcal W_j,\le_j,\nu_j)\) be a standard Borel preordered probability
space, and let
\begin{equation}\label{eq:transition-realization}
 \mathcal T_j:
 \mathcal S_j^\dagger\times\mathcal W_j
 \longrightarrow
 (\mathcal S_{j+1}^\dagger)^{\mathcal C_j},
 \qquad \mathcal T_j\ \text{Borel}.
\end{equation}
Assume
\begin{align}
 \mathcal T_{j,c}(\dagger,w)&=\dagger,\notag\\
 x\preceq_jx'
 &\Longrightarrow
 \mathcal T_{j,c}(x,w)
 \preceq_{j+1}
 \mathcal T_{j,c}(x',w),\notag\\
 w\le_jw'
 &\Longrightarrow
 \mathcal T_{j,c}(x,w')
 \preceq_{j+1}
 \mathcal T_{j,c}(x,w),
 \qquad c\in\mathcal C_j.
 \label{eq:monotonicities}
\end{align}
and
\begin{equation}\label{eq:association-definition}
 \int\prod_{\ell=1}^qF_\ell\,\dd\nu_j
 \ge
 \prod_{\ell=1}^q\int F_\ell\,\dd\nu_j
 \quad
 \substack{q\ge1;\\
 F_\ell:\mathcal W_j\to[0,\infty)
 \text{ bounded, Borel, increasing}.}
\end{equation}

For $q\in\mathbb N_0$, put $[q]:=\{0,\ldots,q-1\}$ and
\[
 \Pop(E):=\bigsqcup_{q\in\mathbb N_0}E^{[q]}.
\]
For $\xi=(q,\mathbf x)\in\Pop(E)$, write $|\xi|:=q$ and
put $0:=(0,\varnothing)$.  Whenever $G$ is defined on $E$, families and
sums indexed by $\xi$ are interpreted entrywise:
\[
 \bigl(G(x)\bigr)_{x\in\xi}
 :=\bigl(G(\mathbf x_r)\bigr)_{r\in[q]},
 \qquad
 \sum_{x\in\xi}G(x)
 :=\sum_{r\in[q]}G(\mathbf x_r).
\]
The order is inherited from $[q]$, and repeated entries remain separately
indexed.
For a finite tuple in $E^\dagger$, let $\operatorname{live}$ delete all
$\dagger$-coordinates while preserving their order.

On a filtered probability space, take
\[
 \Xi_0\in\Pop(\mathcal S_0),\qquad
 \Xi_0\ \mathcal F_0\text{-measurable},\qquad
 \Ee|\Xi_0|<\infty,\qquad
 W_j\sim\nu_j,\qquad
 W_j\perp\!\!\!\perp\mathcal F_j,
\]
\[
 \mathcal F_{j+1}
 :=
 \overline{\sigma(\mathcal F_j,W_j)}^{\Pp}.
\]
If $\Xi_j=(q_j,\mathbf X_j)$, define
\begin{equation}\label{eq:shared-update}
 \Xi_{j+1}
 :=
 \operatorname{live}\!\left(
 \bigl(
 \mathcal T_{j,c}(\mathbf X_{j,r},W_j)
 \bigr)_{(r,c)\in[q_j]\times\mathcal C_j}
 \right).
\end{equation}
Here \([q_j]\times\mathcal C_j\) has the lexicographic order with the parent
index \(r\) first and the slot index \(c\) second.
Since \(|\mathcal C_j|<\infty\), induction gives
\begin{equation}\label{eq:actual-mean}
 m_j:=\Ee|\Xi_j|<\infty.
\end{equation}
Assume \(0\le\delta_j\le1\) and
\begin{equation}\label{eq:abstract-killing}
 \inf_{x\in\mathcal S_j}
 \nu_j\bigl\{w\in\mathcal W_j:
 \mathcal T_{j,c}(x,w)=\dagger\ \text{for every }c\in\mathcal C_j\bigr\}
 \ge\delta_j.
\end{equation}

For \(0\le j\le N\), define
\begin{align}
 u_{N,N}(x)
 &:=\ind_{\{x\ne\dagger\}},\notag\\
 u_{j,N}(x)
 &:=
 1-\int_{\mathcal W_j}
 \prod_{c\in\mathcal C_j}
 \left[1-u_{j+1,N}\bigl(\mathcal T_{j,c}(x,w)\bigr)\right]
 \nu_j(\dd w),
 \qquad j<N.
 \label{eq:bellman}
\end{align}
\begin{lemma}[Properties of the backward recursion]\label{lem:bellman-properties}
Every $u_{j,N}$ is Borel measurable, and
\begin{equation}\label{eq:bellman-properties}
 0\le u_{j,N}\le1,\qquad
 u_{j,N}(\dagger)=0,\qquad
 x\preceq_jx'\Longrightarrow
 u_{j,N}(x)\le u_{j,N}(x').
\end{equation}
\end{lemma}

\begin{proof}
Proceed backward from $j=N$.  The terminal function is Borel, takes values in
$[0,1]$, vanishes at $\dagger$, and is increasing.  Suppose the assertions
hold at time $j+1$.  For fixed $c$, the composition
$u_{j+1,N}\circ\mathcal T_{j,c}$ is Borel; a finite product remains Borel, and
integration against the fixed probability measure $\nu_j$ preserves Borel
measurability in $x$.  The recursion also keeps the values in $[0,1]$ and
vanishes at $\dagger$.  Finally, if $x\preceq_jx'$, then
\[
 u_{j+1,N}(\mathcal T_{j,c}(x,w))
 \le
 u_{j+1,N}(\mathcal T_{j,c}(x',w))
\]
for every $c,w$.  Hence the product of the complementary factors decreases,
so $u_{j,N}(x)\le u_{j,N}(x')$.
\end{proof}

\begin{lemma}[Quadratic union potential]\label{lem:union-potential}
Put
\[
 \varphi(t):=t+\frac12t^2.
\]
For every finite \(\mathcal C\) and \((z_c)_{c\in\mathcal C}\in[0,1]^{\mathcal C}\),
\begin{equation}\label{eq:union-potential}
 \varphi\left(1-\prod_{c\in\mathcal C}(1-z_c)\right)
 \le
 \sum_{c\in\mathcal C}\varphi(z_c).
\end{equation}
\end{lemma}

\begin{proof}
For \(a,b\in[0,1]\),
\[
 \varphi(a)+\varphi(b)-\varphi(a+b-ab)
 =ab\left(a+b-\frac12ab\right)\ge0.
\]
Iterating this inequality gives \eqref{eq:union-potential}.
\end{proof}

\begin{lemma}[One-particle coercivity]\label{lem:coercivity}
For \(0\le j<N\) and \(x\in\mathcal S_j\),
\begin{equation}\label{eq:coercivity}
 \int_{\mathcal W_j}
 \sum_{c\in\mathcal C_j}
 \varphi\!\left(u_{j+1,N}(\mathcal T_{j,c}(x,w))\right)\nu_j(\dd w)
 -\varphi(u_{j,N}(x))
 \ge
 \frac{\delta_j}{2}u_{j,N}(x)^2.
\end{equation}
\end{lemma}

\begin{proof}
Put
\[
 z_c(w):=u_{j+1,N}(\mathcal T_{j,c}(x,w)),\qquad
 R(w):=1-\prod_{c\in\mathcal C_j}(1-z_c(w)),
\]
\[
 s:=u_{j,N}(x)=\int R\,\dd\nu_j,
 \qquad
 D:=\{w:R(w)=0\}.
\]
By \eqref{eq:abstract-killing}, \(\nu_j(D)\ge\delta_j\).  Hence
\[
 s^2
 =\left(\int_{D^c}R\,\dd\nu_j\right)^2
 \le\nu_j(D^c)\int R^2\,\dd\nu_j
 \le(1-\delta_j)\int R^2\,\dd\nu_j,
\]
and therefore
\begin{equation}\label{eq:killed-variance}
 \int R^2\,\dd\nu_j-s^2\ge\delta_js^2.
\end{equation}
By \Cref{lem:union-potential},
\begin{align*}
 \int\sum_c\varphi(z_c)\,\dd\nu_j-\varphi(s)
 &\ge\int\varphi(R)\,\dd\nu_j
 -\varphi\!\left(\int R\,\dd\nu_j\right)\\
 &=\frac12\left(\int R^2\,\dd\nu_j-s^2\right)
 \ge\frac{\delta_j}{2}s^2.
\end{align*}
\end{proof}

For
\[
 \xi=(q,\mathbf x)\in\Pop(\mathcal S_j),
\]
define
\begin{equation}\label{eq:bellman-population}
 V_{j,N}(\xi)
 :=
 1-\prod_{r\in[q]}\bigl(1-u_{j,N}(\mathbf x_r)\bigr).
\end{equation}

\begin{proposition}[Positive-association comparison]\label{prop:bellman-fkg}
For \(0\le j\le N\),
\begin{equation}\label{eq:bellman-fkg}
 \Ee\!\left[
 \ind_{\{\Xi_N\ne0\}}
 \mid\mathcal F_j
 \right]
 \le
 V_{j,N}(\Xi_j)
 \qquad\text{a.s.}
\end{equation}
\end{proposition}

\begin{proof}
Condition on $\mathcal F_j$.  The current population is then fixed, whereas
$W_j$ has law $\nu_j$.  If
\(\Xi_j=(q,\mathbf x)\), put
\[
 g_{j,r}(w)
 :=
 \prod_{c\in\mathcal C_j}
 \left[
 1-u_{j+1,N}
 \bigl(\mathcal T_{j,c}(\mathbf x_r,w)\bigr)
 \right].
\]
By \eqref{eq:monotonicities} and
\eqref{eq:bellman-properties}, \(g_{j,r}\uparrow\) in \(w\).  Hence
\begin{align*}
 \Ee\!\left[
 \prod_{r\in[q]}g_{j,r}(W_j)
 \mid\mathcal F_j
 \right]
 &\ge
 \prod_{r\in[q]}
 \int g_{j,r}\,\dd\nu_j\\
 &=
 \prod_{r\in[q]}
 \bigl(1-u_{j,N}(\mathbf x_r)\bigr),
 \end{align*}
Since \(u_{j+1,N}(\dagger)=0\), the two sides of the last inequality are,
respectively,
\[
 1-\Ee\!\left[V_{j+1,N}(\Xi_{j+1})\mid\mathcal F_j\right],
 \qquad
 1-V_{j,N}(\Xi_j).
\]
Therefore
\[
 \Ee\!\left[
 V_{j+1,N}(\Xi_{j+1})
 \mid\mathcal F_j
 \right]
 \le
 V_{j,N}(\Xi_j).
\]
Since
\[
 V_{N,N}(\Xi_N)=\ind_{\{\Xi_N\ne0\}},
\]
backward iteration gives \eqref{eq:bellman-fkg}.
\end{proof}

\begin{theorem}[Survival bound and extinction criterion]\label{thm:shared-resistance}
If \(m_j>0\) for \(0\le j\le N\), then
\begin{equation}\label{eq:shared-resistance}
 \Pp(\Xi_N\ne0)
 \le
 \left[
 \frac{2}{3m_N}
 +\frac16\sum_{j<N}\frac{\delta_j}{m_j}
 \right]^{-1}.
\end{equation}
Consequently,
\begin{equation}\label{eq:shared-extinction}
 \left[
 m_j>0\ (j\ge0),\qquad
 \sum_{j\ge0}\frac{\delta_j}{m_j}=\infty
 \right]
 \Longrightarrow
 \Pp(\Xi_N\ne0)\longrightarrow0.
\end{equation}
If, in addition, $(\mathfrak R_j)_{j\ge0}$ is a sequence of random sets such
that, for every $j\ge0$,
\[
 \mathfrak R_{j+1}\subseteq\mathfrak R_j\quad\text{a.s.},\qquad
 \{\Xi_j\ne0\}=\{\mathfrak R_j\ne\varnothing\},
\]
then
\begin{equation}\label{eq:residual-extinction}
 \Pp(\exists N:\mathfrak R_N=\varnothing)=1.
\end{equation}
\end{theorem}

\begin{proof}
Put
\[
 C_j^{(N)}
 :=
 \Ee\sum_{x\in\Xi_j}\varphi(u_{j,N}(x)),
\]
By \eqref{eq:bellman-fkg},
\begin{equation}\label{eq:initial-conductance}
 \Pp(\Xi_N\ne0)
 \le\Ee V_{0,N}(\Xi_0)
 \le\Ee\sum_{x\in\Xi_0}u_{0,N}(x)
 \le C_0^{(N)}.
\end{equation}
To pass from the one-particle estimate to the population, condition on
$\mathcal F_j$ and use \eqref{eq:shared-update}.  Linearity of the sum is
available even though all particles use the same $W_j$:
\begin{align*}
 &\Ee\!\left[
 \sum_{y\in\Xi_{j+1}}\varphi(u_{j+1,N}(y))
 \middle|\mathcal F_j\right]\\
 &\qquad=
 \sum_{x\in\Xi_j}\int_{\mathcal W_j}
 \sum_{c\in\mathcal C_j}
 \varphi\!\left(u_{j+1,N}(\mathcal T_{j,c}(x,w))\right)
 \nu_j(\dd w).
\end{align*}
Applying \Cref{lem:coercivity} to each current entry and then taking
expectations gives
\begin{equation}\label{eq:conductance-increment}
 C_{j+1}^{(N)}-C_j^{(N)}
 \ge
 \frac{\delta_j}{2}
 \Ee\sum_{x\in\Xi_j}u_{j,N}(x)^2.
\end{equation}
Moreover,
\[
 \sum_{x\in\Xi_j}u_{j,N}(x)
 \le
 |\Xi_j|^{1/2}
 \left(\sum_{x\in\Xi_j}u_{j,N}(x)^2\right)^{1/2}.
\]
Cauchy--Schwarz and \(u\le\varphi(u)\le\frac32u\) therefore yield
\begin{equation}\label{eq:population-cs}
 \bigl(C_j^{(N)}\bigr)^2
 \le
 \frac94m_j\Ee\sum_{x\in\Xi_j}u_{j,N}(x)^2.
\end{equation}
Since \(0\le u_{j,N}\le1\),
\begin{equation}\label{eq:conductance-linear-bounds}
 C_j^{(N)}\le\frac32m_j.
\end{equation}
Thus
\begin{equation}\label{eq:conductance-growth}
 C_{j+1}^{(N)}-C_j^{(N)}
 \ge
 \frac29\frac{\delta_j}{m_j}
 \bigl(C_j^{(N)}\bigr)^2.
\end{equation}
If \(C_0^{(N)}=0\), then \eqref{eq:initial-conductance} proves the result.
Otherwise \eqref{eq:conductance-growth} gives \(C_j^{(N)}>0\) for every
\(j\le N\), while \eqref{eq:conductance-linear-bounds} gives
\[
 \frac{\delta_j}{m_j}C_j^{(N)}\le\frac32.
\]
Write
$a_j=(2/9)(\delta_j/m_j)$ and
$D_j=C_{j+1}^{(N)}-C_j^{(N)}$.  Since
$D_j\ge a_j(C_j^{(N)})^2$ and the function
$D\mapsto D/[C_j^{(N)}(C_j^{(N)}+D)]$ is increasing,
\begin{align*}
 \frac1{C_j^{(N)}}-\frac1{C_{j+1}^{(N)}}
 &\ge
 \frac{(2/9)\delta_j/m_j}
 {1+(2/9)(\delta_j/m_j)C_j^{(N)}}\\
 &\ge\frac16\frac{\delta_j}{m_j}.
\end{align*}
Hence
\[
 \frac1{C_0^{(N)}}
 \ge
 \frac1{C_N^{(N)}}
 +\frac16\sum_{j<N}\frac{\delta_j}{m_j},
 \qquad
 C_N^{(N)}=\frac32m_N.
\]
Together with \eqref{eq:initial-conductance}, this proves
\eqref{eq:shared-resistance}.  Finally,
\[
 \Pp\!\left(\bigcap_N\{\mathfrak R_N\ne\varnothing\}\right)
 =
 \lim_N\Pp(\Xi_N\ne0)
 =0.
\]
\end{proof}

\section{Application to uncovered dyadic cells}
\label{sec:spatial-application}

We now verify, one hypothesis at a time, that the uncovered-cell process fits
the abstract recursion of the preceding section.  A deletion step reveals one
Poisson block and has one output slot; a refinement step is deterministic and
has $2^d$ output slots.  The expected population used below is the actual mean
number of cells meeting the uncovered set, not an auxiliary scheduled mean.

For \(k\ge k_*\) and \(0\le p\le M_k\), put
\begin{equation}\label{eq:time-embedding}
 \tau(k,p)
 :=
 \sum_{r=k_*}^{k-1}(M_r+1)+p.
\end{equation}
\[
 \tau(k,p-1)
 \xrightarrow{\operatorname{delete}(k,p)}
 \tau(k,p)\quad(1\le p\le M_k),
 \qquad
 \tau(k,M_k)
 \xrightarrow{\operatorname{refine}(k)}
 \tau(k+1,0).
\]

Let \(\mathcal K_+(Q)\) be the nonempty compact subsets of \(Q\), with the
Hausdorff Borel structure, and define
\begin{equation}\label{eq:spatial-state}
 \mathcal S_{\tau(k,p)}
 :=
 \bigsqcup_{\alpha\in\mathcal I_k}
 \bigl(\{\alpha\}\times\mathcal K_+(Q_{k,\alpha})\bigr).
\end{equation}
Adjoin \(\dagger\) and order each space by
\[
 \dagger\preceq_{\tau(k,p)}(\alpha,A),
 \qquad
 (\alpha,A)\preceq_{\tau(k,p)}(\alpha',A')
 \Longleftrightarrow
 \alpha=\alpha'\ \text{and}\ A\subseteq A'.
\]

Fix a total order on $\mathcal I_{k_*}$ and a total order on
$\{0,1\}^d$.  Recursively order $\mathcal I_{k+1}$ through the bijection
\[
 \mathcal I_k\times\{0,1\}^d\longrightarrow\mathcal I_{k+1},
 \qquad
 (\alpha,\varepsilon)\longmapsto
 \operatorname{ch}_k(\alpha,\varepsilon),
\]
with the parent coordinate first.  Give every subset of $\mathcal I_k$ the
induced dyadic-tree order.  Define
\begin{equation}\label{eq:spatial-Xi}
\begin{aligned}
 \mathcal A_{k,p}
 &:=
 \{\alpha\in\mathcal I_k:
 \mathfrak R_{k,p}\cap Q_{k,\alpha}\ne\varnothing\},\\
 \Xi_{k,p}
 &:=
 \bigl(
 |\mathcal A_{k,p}|,
 \bigl((\alpha,\mathfrak R_{k,p}\cap Q_{k,\alpha})\bigr)_
      {\alpha\in\mathcal A_{k,p}}^{\uparrow}
 \bigr)\in\Pop(\mathcal S_{\tau(k,p)}),\\
 \qquad
 \Xi_{\tau(k,p)}&:=\Xi_{k,p},
 \qquad
 \mathfrak R_{\tau(k,p)}:=\mathfrak R_{k,p}.
\end{aligned}
\end{equation}
By \eqref{eq:poisson-superposition},
\begin{equation}\label{eq:packet-boundary}
 \mathfrak R_{k,M_k}=\mathfrak R_{k+1,0}.
\end{equation}
Consequently,
\begin{equation}\label{eq:survival-identification}
 \Xi_{\tau(k,p)}\ne0
 \quad\Longleftrightarrow\quad
 \mathfrak R_{k,p}\ne\varnothing.
\end{equation}

At a deletion source and a refinement source, respectively, set
\begin{align}
 \mathcal W_{\tau(k,p-1)}
 &:=
 \Cloud_{\mathcal N_k}(\T),
 &
 W_{\tau(k,p-1)}
 &:=
 \boldsymbol\Pi_{k,p},
 \qquad 1\le p\le M_k,
 \label{eq:deletion-innovation}\\
 \mathcal W_{\tau(k,M_k)}
 &:=
 \{\star\},
 &
 W_{\tau(k,M_k)}
 &:=
 \star.
 \label{eq:refinement-innovation}
\end{align}
Let
\[
 \nu_j:=\mathcal L(W_j),
 \qquad
 \mathcal F_j
 :=
 \overline{\sigma(W_i:0\le i<j)}^{\Pp}.
\]
Then
\begin{equation}\label{eq:spatial-association}
 W_j\perp\!\!\!\perp\mathcal F_j,
 \qquad
 \nu_j\ \text{is positively associated}.
\end{equation}
At deletion times this follows from \Cref{lem:cloud-association}; at refinement
times $\nu_j$ is a point mass, so the same inequality is immediate.

Use the variable slot sets
\begin{equation}\label{eq:spatial-slots}
 \mathcal C_{\tau(k,p-1)}:=\{\star\},
 \qquad
 \mathcal C_{\tau(k,M_k)}:=\{0,1\}^d.
\end{equation}
The singleton has its unique order, and \(\{0,1\}^d\) has the fixed order
chosen immediately before \eqref{eq:spatial-Xi}.
The deletion transition is
\begin{equation}\label{eq:deletion-transition}
 \mathcal T_{k,p}^{\mathrm{del}}:
 \mathcal S_{\tau(k,p-1)}^\dagger
 \times\Cloud_{\mathcal N_k}(\T)
 \longrightarrow
 (\mathcal S_{\tau(k,p)}^\dagger)^{\{\star\}},
\end{equation}
\begin{equation}\label{eq:deletion-map}
 \mathcal T_{k,p,\star}^{\mathrm{del}}(x,\omega)
 :=
 \begin{cases}
 (\alpha,A\setminus\Cov(\omega)),
 &x=(\alpha,A),\
 A\setminus\Cov(\omega)\ne\varnothing,\\
 \dagger,&\text{otherwise}.
 \end{cases}
\end{equation}
The refinement transition is
\begin{equation}\label{eq:refinement-transition}
 \mathcal T_k^{\mathrm{ref}}:
 \mathcal S_{\tau(k,M_k)}^\dagger
 \times\{\star\}
 \longrightarrow
 (\mathcal S_{\tau(k+1,0)}^\dagger)^{\{0,1\}^d},
\end{equation}
\begin{equation}\label{eq:refinement-map}
 \mathcal T_{k,\varepsilon}^{\mathrm{ref}}(x,\star)
 :=
 \begin{cases}
 \bigl(
 \operatorname{ch}_k(\alpha,\varepsilon),
 A\cap Q_{k+1,\operatorname{ch}_k(\alpha,\varepsilon)}
 \bigr),
 &\begin{gathered}
 x=(\alpha,A),\\[-1mm]
 A\cap Q_{k+1,\operatorname{ch}_k(\alpha,\varepsilon)}
 \ne\varnothing,
 \end{gathered}\\
 \dagger,&\text{otherwise}.
 \end{cases}
\end{equation}
With the active-point preorder
\[
 \omega\le\omega'
 \quad\Longleftrightarrow\quad
 [\omega_n]\subseteq[\omega'_n]
 \quad(n\in\mathcal N_k),
\]
one has
\begin{align}
 A\subseteq A'
 &\Longrightarrow
 A\setminus\Cov(\omega)
 \subseteq
 A'\setminus\Cov(\omega),
 \notag\\
 \omega\le\omega'
 &\Longrightarrow
 A\setminus\Cov(\omega')
 \subseteq
 A\setminus\Cov(\omega),
 \notag\\
 A\subseteq A'
 &\Longrightarrow
 A\cap Q_{k+1,\beta}
 \subseteq
 A'\cap Q_{k+1,\beta}.
 \label{eq:spatial-monotonicity}
\end{align}
Thus \eqref{eq:monotonicities} holds.  The transition maps are Borel by the
hyperspace measurability argument in \Cref{app:measurability}, so
\eqref{eq:transition-realization} is also satisfied.

Moreover,
\begin{align}
 \Xi_{\tau(k,p)}
 &=
 \operatorname{live}\!\left(
 \bigl(
 \mathcal T_{k,p,\star}^{\mathrm{del}}
 (x,\boldsymbol\Pi_{k,p})
 \bigr)_{x\in\Xi_{\tau(k,p-1)}}
 \right),
 \label{eq:spatial-deletion-update}\\
 \Xi_{\tau(k+1,0)}
 &=
 \operatorname{live}\!\left(
 \bigl(
 \mathcal T_{k,\varepsilon}^{\mathrm{ref}}(x,\star)
 \bigr)_{\substack{x\in\Xi_{\tau(k,M_k)}\\
                    \varepsilon\in\{0,1\}^d}}
 \right).
 \label{eq:spatial-refinement-update}
\end{align}
In \eqref{eq:spatial-refinement-update}, the pairs $(x,\varepsilon)$ are
ordered with the parent coordinate first and the fixed child order second.
Hence \((\Xi_j)_j\) is exactly the process in
\eqref{eq:shared-update}, with
\[
 \Xi_0
 =
 \left(
 |\mathcal I_{k_*}|,
 \bigl((\alpha,Q_{k_*,\alpha})\bigr)_{\alpha\in\mathcal I_{k_*}}^\uparrow
 \right).
\]

At deletion and refinement sources, define
\begin{equation}\label{eq:spatial-resistance-data}
 \begin{aligned}
 m_{\tau(k,p-1)}
 &:=
 \Ee|\Xi_{k,p-1}|
 =
 m_{k,p-1},
 &
 \delta_{\tau(k,p-1)}
 &:=
 (1-\e^{-q_d})b_{k,p},
 \\
 m_{\tau(k,M_k)}
 &:=
 \Ee|\Xi_{k,M_k}|,
 &
 \delta_{\tau(k,M_k)}
 &:=
 0.
 \end{aligned}
\end{equation}
By \Cref{lem:one-cell},
\begin{align}
 m_{\tau(k,p-1)}
 &\le
 C_d\ell_k^{-d}\e^{-a_{k,p-1}},
 \notag\\
 \inf_{x\in\mathcal S_{\tau(k,p-1)}}
 \nu_{\tau(k,p-1)}
 \bigl\{w:\mathcal T_{k,p,\star}^{\mathrm{del}}(x,w)=\dagger\bigr\}
 &\ge
 \delta_{\tau(k,p-1)}.
 \label{eq:spatial-one-cell}
\end{align}
This is precisely \eqref{eq:abstract-killing} at deletion sources.  At a
refinement source we chose $\delta_j=0$, for which the same condition is
automatic.  Hence all assumptions of
\Cref{thm:shared-resistance} have now been verified for the process
$(\Xi_j)_j$.
Therefore, if all \(m_j>0\),
\begin{align*}
 \sum_{j\ge0}\frac{\delta_j}{m_j}
 &=
 (1-\e^{-q_d})
 \sum_{k\ge k_*}\sum_{p=1}^{M_k}
 \frac{b_{k,p}}{m_{k,p-1}}
 \notag\\
 &\ge
 c_d\sum_{k\ge k_*}\ell_k^d
 \sum_{p=1}^{M_k}
 b_{k,p}\e^{a_{k,p-1}}
 \notag\\
 &\ge
 \frac{c_d}{\e-1}
 \sum_{k\ge k_*}\ell_k^d
 \sum_{p=1}^{M_k}
 \bigl(
 \e^{a_{k,p}}-\e^{a_{k,p-1}}
 \bigr)
 \notag\\
 &=
 \frac{c_d}{\e-1}
 \sum_{k\ge k_*}\ell_k^d
 \bigl(
 \e^{\whatA_k}-\e^{\whatA_{k-1}}
 \bigr).
\end{align*}
Here
\[
 0\le b\le1
 \Longrightarrow
 \e^{a+b}-\e^a
 \le(\e-1)b\e^a,
\]
Consequently, \Cref{lem:analytic-mass} gives
\begin{equation}\label{eq:actual-resistance-diverges}
 \left[
 m_j>0\ (j\ge0),\qquad
 \int_{\T}\e^H\,\dd m=\infty
 \right]
 \Longrightarrow
 \sum_{j\ge0}\frac{\delta_j}{m_j}=\infty.
\end{equation}

\begin{theorem}[Poissonized sufficiency]\label{thm:poisson-sufficiency}
If \(\int_{\T}\e^H\,\dd m=\infty\), then
\begin{equation}\label{eq:poisson-tail-cover}
 \Pp\!\left(
 \forall n_0\ \exists n_1\ge n_0:\
 \T\subseteq
 \bigcup_{n=n_0}^{n_1}
 \Cov(\Pi_n^{\mathrm{full}})
 \right)=1.
\end{equation}
\end{theorem}

\begin{proof}
Fix $n_0$ and reindex the deterministic tail by
\[
 r^{[n_0]}_j:=r_{n_0+j-1},\qquad
 \Pi^{[n_0]}_j:=\Pi_{n_0+j-1}^{\mathrm{full}},\qquad j\ge1.
\]
Its overlap functions and accumulated overlap are
\[
 u^{[n_0]}_j(z)
 :=m\bigl(B(0,r^{[n_0]}_j)\cap B(z,r^{[n_0]}_j)\bigr),
 \qquad
 H^{[n_0]}(z)
 :=\sum_{j\ge1}u^{[n_0]}_j(z)
 =\sum_{n\ge n_0}u_n(z).
\]
Since
\[
 0\le H(z)-H^{[n_0]}(z)
 \le\sum_{n<n_0}v_n<\infty,
\]
the integrals of $\e^H$ and $\e^{H^{[n_0]}}$ diverge simultaneously.  Apply
the constructions of
\Cref{sec:analytic-localization,sec:poisson-decomposition,sec:spatial-application}
to $(r^{[n_0]}_j)_{j\ge1}$, decorating all
scale, block, residual-set, and population variables by $[n_0]$.  If one of
its mean populations vanishes, then
\[
 \bigl[\exists j:m_j^{[n_0]}=0\bigr]
 \Longrightarrow
 \Pp(\Xi_j^{[n_0]}=0)=1.
\]
Otherwise, the divergence estimate and the abstract extinction criterion give
\[
 \int_{\T}\e^{H^{[n_0]}}\,\dd m=\infty
 \overset{\eqref{eq:actual-resistance-diverges}}{\Longrightarrow}
 \sum_j\frac{\delta_j^{[n_0]}}{m_j^{[n_0]}}=\infty
 \overset{\text{\Cref{thm:shared-resistance}}}{\Longrightarrow}
 \Pp(\exists j:\mathfrak R_j^{[n_0]}=\varnothing)=1.
\]
In both cases the tail is covered after finitely many of its radius types.
By \eqref{eq:poisson-superposition}, $(\Pi^{[n_0]}_j)_{j\ge1}$ has exactly
the law used in this tail construction.  Therefore
\[
 \Pp\!\left(
 \exists J\ge1:\
 \T\subseteq
 \bigcup_{j=1}^{J}\Cov(\Pi^{[n_0]}_j)
 \right)=1.
\]
Writing $n_1=n_0+J-1$ gives the event in
\eqref{eq:poisson-tail-cover} for this $n_0$.
The conclusion holds for each deterministic $n_0$; intersecting the resulting
probability-one events over \(n_0\in\N\) proves
\eqref{eq:poisson-tail-cover}.
\end{proof}

\section{De-Poissonization}\label{sec:depoissonization}

The Poissonized model may use more than one center at a given radius type,
whereas the original model has exactly one center for each index.  We reserve
a sparse deterministic set of original indices and use the remaining indices
to absorb the Poisson multiplicities.  The high-moment regime is disposed of
first by a direct net argument; in the complementary regime the reserved
indices have finite total volume cost.

\begin{lemma}[Borel tail property of full limsup coverage]
\label{lem:limsup-cover-borel-tail}
The subset
\[
 \left\{(x_n)_{n\ge1}\in(\T)^{\N}:
 \limsup_{n\to\infty}B(x_n,r_n)=\T\right\}
\]
of $(\T)^{\N}$ is Borel and belongs to the coordinate tail $\sigma$-field.
\end{lemma}

\begin{proof}
Compactness gives
\begin{equation}\label{eq:tail-event}
\begin{aligned}
 &\left\{(x_n)_{n\ge1}\in(\T)^{\N}:
   \limsup_{n\to\infty}B(x_n,r_n)=\T\right\}\\
 &\qquad=
 \bigcap_{m\ge1}\ \bigcup_{N\ge m}
 \left\{(x_n)_{n\ge1}:
 \T\subseteq\bigcup_{n=m}^{N}B(x_n,r_n)\right\}.
\end{aligned}
\end{equation}
For fixed $1\le m\le N$, define
\[
 \begin{aligned}
 \Gamma_{m,N}:(\T)^{N-m+1}&\longrightarrow\R,\\
 (x_m,\ldots,x_N)&\longmapsto
 \max_{y\in\T}\min_{m\le n\le N}
 \bigl(d_{\T}(y,x_n)-r_n\bigr).
 \end{aligned}
\]
The map $\Gamma_{m,N}$ is continuous, and
\[
 \begin{aligned}
 &\left\{(x_n)_{n\ge1}\in(\T)^{\N}:
 \T\subseteq\bigcup_{n=m}^{N}B(x_n,r_n)\right\}\\
 &\qquad=
 \left\{(x_n)_{n\ge1}\in(\T)^{\N}:
 \Gamma_{m,N}(x_m,\ldots,x_N)<0\right\}
 \in\mathcal B((\T)^{\N}).
 \end{aligned}
\]
Thus \eqref{eq:tail-event} is Borel.  For every $s\ge1$,
the right-hand side is unchanged if the outer intersection is restricted to
$m\ge s$; it then depends only on coordinates with index at least $s$.
Hence the event belongs to the coordinate tail $\sigma$-field.
\end{proof}

\begin{lemma}[High-moment regime]\label{lem:high-moment}
\[
 \left(\exists\varepsilon>0:\sum_nv_n^{1+\varepsilon}=\infty\right)
 \Longrightarrow \Pp(E_r=\T)=1.
\]
\end{lemma}

\begin{proof}
Fix $\varepsilon>0$ such that the displayed series diverges, choose
$0<\gamma<\varepsilon$, and set
\[
 \mathcal K_\varepsilon
 :=\{k\in\N:v_k^{1+\varepsilon}\ge k^{-(1+\gamma)}\}.
\]
Then $|\mathcal K_\varepsilon|=\infty$.  For $k\in\mathcal K_\varepsilon$,
choose
\[
 \mathcal G_k\subseteq\T,\qquad
 \T=\bigcup_{y\in\mathcal G_k}B(y,r_k/2),\qquad
 |\mathcal G_k|\le C_dv_k^{-1}.
\]
Such a set is obtained from a maximal $r_k/2$-separated set; the disjoint
$r_k/4$-balls give the cardinality bound.  If
$\mathcal K_\varepsilon$ were finite, then eventually
$v_k^{1+\varepsilon}<k^{-(1+\gamma)}$, contradicting divergence of the
series in the hypothesis.
For each fixed $n_0\in\N$, there exists $k_0=k_0(n_0)$ such that, for every
$k\in\mathcal K_\varepsilon$ with $k\ge k_0$,
\begin{align*}
 \Pp\!\left(\mathcal G_k\not\subseteq
 \bigcup_{j=n_0}^{k}B(X_j,r_j/2)\right)
 &\le C_dv_k^{-1}
 \exp[-2^{-d}(k-n_0+1)v_k]\\
 &\le C_dk^{(1+\gamma)/(1+\varepsilon)}
 \exp[-2^{-d-1}k^{(\varepsilon-\gamma)/(1+\varepsilon)}].
\end{align*}
Moreover,
\[
 \mathcal G_k\subseteq\bigcup_{j=n_0}^{k}B(X_j,r_j/2)
 \Longrightarrow
 \T\subseteq\bigcup_{j=n_0}^{k}B(X_j,r_j),
\]
because $r_j\ge r_k$ for $j\le k$.  Borel--Cantelli, followed by the
intersection over $n_0\in\N$, gives $\Pp(E_r=\T)=1$.
\end{proof}

Assume henceforth
\[
 V_{4/3}:=\sum_nv_n^{4/3}<\infty,\qquad
 \zeta_q:=\lceil q^{3/2}\rceil,\qquad
 \mathcal R:=\{\zeta_q:q\ge1\}.
\]
The remaining construction follows the route
\[
 \begin{aligned}
 \text{sparse reserve}
 &\longrightarrow \text{retained Poisson clouds}
 \longrightarrow \text{ranked i.i.d. centers}\\
 &\longrightarrow \text{eventual rank domination}
 \longrightarrow \text{original covering}.
 \end{aligned}
\]
Then
\begin{align}
 |\mathcal R\cap[1,N]|&=\lfloor N^{2/3}\rfloor,\label{eq:reserve-count}\\
 v_n&\le V_{4/3}^{3/4}n^{-3/4},\notag\\
 \Lambda_{\mathcal R}:=\sum_{n\in\mathcal R}v_n
 &\le V_{4/3}^{3/4}\sum_{q\ge1}q^{-9/8}<\infty.\label{eq:reserve-cost}
\end{align}
For the second estimate, monotonicity of $(v_n)_n$ gives
$nv_n^{4/3}\le\sum_{j\le n}v_j^{4/3}\le V_{4/3}$.
For $H^{\mathrm{ret}}:=\sum_{n\notin\mathcal R}u_n$,
\begin{equation}\label{eq:energy-retained}
 H^{\mathrm{ret}}\ge H-\Lambda_{\mathcal R},\qquad
 \int\e^H=\infty\Longrightarrow\int\e^{H^{\mathrm{ret}}}=\infty.
\end{equation}
Let $(\widetilde P_n)_{n\notin\mathcal R}$ be independent with
$\widetilde P_n\sim\Pois(1)$, and put $\widetilde P_n:=0$ for
$n\in\mathcal R$.  Then
\[
 P_N^{\mathrm{ret}}:=\sum_{n\le N}\widetilde P_n
 \sim\Pois(N-\lfloor N^{2/3}\rfloor)
\]
and, with
$\mu_N=N-\lfloor N^{2/3}\rfloor$ and
$t_N=N-\mu_N$,
\[
 \Pp(P_N^{\mathrm{ret}}>N)
 \le
 \exp\!\left[-\frac{t_N^2}{2(\mu_N+t_N/3)}\right]
 \le\e^{-cN^{1/3}}.
\]

Take $(U_j)_{j\ge1}\stackrel{\mathrm{iid}}{\sim}m$ independently of
$(\widetilde P_n)_n$, and set
\[
 E_r^U:=\bigcap_{M\ge1}\bigcup_{j\ge M}B(U_j,r_j).
\]
For
$n\notin\mathcal R$, $1\le q\le\widetilde P_n$,
\begin{equation}\label{eq:rank}
 \operatorname{rk}(n,q):=\sum_{m<n}\widetilde P_m+q.
\end{equation}
This map is injective on the admissible pairs.  Indeed, if $n<n'$, then
for $1\le q\le\widetilde P_n$ and
$1\le q'\le\widetilde P_{n'}$,
\[
 \operatorname{rk}(n,q)
 \le \sum_{m\le n}\widetilde P_m
 \le \sum_{m<n'}\widetilde P_m
 < \operatorname{rk}(n',q'),
\]
while injectivity for fixed $n$ follows directly from the definition.
Conditional on the counts,
\[
 (U_{\operatorname{rk}(n,q)})_{n,q}\stackrel{\mathrm{iid}}{\sim}m.
\]
Coverage depends only on the active prefix of a cloud.  Accordingly, for
$\omega=(q,\mathbf x)\in\Cloud(\T)$, put
\[
 \operatorname{act}(\omega)
 :=(q,\mathbf x|_{[q]})\in\Pop(\T).
\]
Define the retained active populations and their covering sets by
\begin{equation}\label{eq:retained-cover-sets}
 \begin{aligned}
 \widetilde\Gamma_n
 &:={}
 \left(
 \widetilde P_n,
 \left(
 U_{\sum_{m<n}\widetilde P_m+i+1}
 \right)_{i\in[\widetilde P_n]}
 \right)
 \in\Pop(\T),\\
 \widetilde C_n
 &:={}
 \bigcup_{i\in[\widetilde P_n]}
 B\!\left(U_{\sum_{m<n}\widetilde P_m+i+1},r_n\right).
 \end{aligned}
\end{equation}
Let
\[
  \iota(1)<\iota(2)<\cdots,\qquad
 \{\iota(q):q\ge1\}=\N\setminus\mathcal R,
\]
be the increasing enumeration of the retained indices, and put
$r'_q:=r_{\iota(q)}$.  Let
$(\Pi_q^{\prime,\mathrm{full}})_{q\ge1}$ denote the independent rate-one
Poisson clouds in the construction of \Cref{thm:poisson-sufficiency} for
the radius sequence $(r'_q)_q$.  For every finite $J\subset\N$, conditioning
on the counts and using injectivity of $\operatorname{rk}$ gives
\[
 \bigl(\widetilde\Gamma_{\iota(q)}\bigr)_{q\in J}
 \stackrel{\mathrm d}{=}
 \bigl(\operatorname{act}(\Pi_q^{\prime,\mathrm{full}})\bigr)_{q\in J}.
\]
Since $\Pop(\T)$ is standard Borel, the finite-dimensional identities give
\begin{equation}\label{eq:retained-cover-law}
 \bigl(\widetilde\Gamma_{\iota(q)}\bigr)_{q\ge1}
 \stackrel{\mathrm d}{=}
 \bigl(\operatorname{act}(\Pi_q^{\prime,\mathrm{full}})\bigr)_{q\ge1}.
\end{equation}
The map \(\operatorname{act}:\Cloud(\T)\to\Pop(\T)\) is Borel: for
\(n\in\mathbb N_0\) and \(A\in\mathcal B((\T)^{[n]})\),
\[
 \operatorname{act}^{-1}(\{n\}\times A)
 =
 \{n\}\times
 \{\mathbf x\in(\T)^{\mathbb N}:\mathbf x|_{[n]}\in A\}.
\]
For \(\boldsymbol\gamma=(\gamma_q)_{q\ge1}\in\Pop(\T)^{\mathbb N}\), write
\(\gamma_q=(n_q,\mathbf x_q)\) and set
\[
 C_q^{r'}(\boldsymbol\gamma)
 :=
 \bigcup_{i\in[n_q]}B(\mathbf x_q(i),r'_q).
\]
Compactness of \(\T\) gives
\[
 \mathcal E_{r'}
 :=
 \left\{\boldsymbol\gamma:
 \limsup_{q\to\infty}C_q^{r'}(\boldsymbol\gamma)=\T\right\}
 =
 \bigcap_{M\ge1}\bigcup_{N\ge M}
 \left\{\boldsymbol\gamma:
 \T\subseteq\bigcup_{q=M}^{N}C_q^{r'}(\boldsymbol\gamma)\right\}.
\]
Fix \(M\le N\) and a count vector
\(\mathbf n=(n_M,\ldots,n_N)\in\mathbb N_0^{N-M+1}\).  If
\(\sum_{q=M}^Nn_q>0\), then on this count stratum the finite-cover event is
\[
 \Phi_{M,N,\mathbf n}((x_{q,i}))<0,
 \qquad
 \Phi_{M,N,\mathbf n}((x_{q,i}))
 :=
 \max_{z\in\T}
 \min_{\substack{M\le q\le N\\ i\in[n_q]}}
 \bigl(d_{\T}(z,x_{q,i})-r'_q\bigr).
\]
The function \(\Phi_{M,N,\mathbf n}\) is continuous; for the all-zero count
vector the finite-cover event is empty.  The count strata form a countable
partition, hence \(\mathcal E_{r'}\) is Borel.  Therefore
\eqref{eq:retained-cover-law} transfers this event:
\[
 \Pp\!\left((\widetilde\Gamma_{\iota(q)})_{q\ge1}\in\mathcal E_{r'}\right)
 =
 \Pp\!\left(
 (\operatorname{act}(\Pi_q^{\prime,\mathrm{full}}))_{q\ge1}
 \in\mathcal E_{r'}\right).
\]
Moreover, $r'_q\downarrow0$, and its overlap function is
\[
 H'(z)
 :=\sum_{q\ge1}m(B(0,r'_q)\cap B(z,r'_q))
 =\sum_{n\notin\mathcal R}u_n(z)
 =H^{\mathrm{ret}}(z).
\]
Consequently, applying \Cref{thm:poisson-sufficiency} to the relabelled
sequence $(r'_q)_q$, then using \eqref{eq:retained-cover-law} and
\eqref{eq:energy-retained}, gives
\begin{equation}\label{eq:retained-poisson-cover}
 \int_{\T}\e^H\,\dd m=\infty
 \quad\Longrightarrow\quad
 \Pp\!\left(
 \limsup_{q\to\infty}\widetilde C_{\iota(q)}=\T
 \right)
 =1.
\end{equation}
Let
\[
 \Omega_0
 :=\{\exists N_0\ \forall N\ge N_0:\ P_N^{\mathrm{ret}}\le N\}.
\]
Then $\Pp(\Omega_0)=1$.  If $\omega\in\Omega_0$ and $N_0=N_0(\omega)$
witnesses the defining quantifier, then for every $n\ge N_0$ with
$n\notin\mathcal R$ and every $1\le q\le\widetilde P_n(\omega)$, with
$j:=\operatorname{rk}(n,q)$,
\[
 j\le P_n^{\mathrm{ret}}\le n,\qquad
 r_j\ge r_n,\qquad
 B(U_j,r_n)\subseteq B(U_j,r_j).
\]
By injectivity of $\operatorname{rk}$, distinct retained population entries
correspond to distinct original indices $j$.  Therefore
\begin{equation}\label{eq:rank-inclusion}
 \limsup_{\substack{n\to\infty\\n\notin\mathcal R}}
 \widetilde C_n
 \subseteq E_r^U\qquad\text{on }\Omega_0.
\end{equation}

\begin{theorem}[Sufficiency]\label{thm:sufficiency}
\[
 \int_{\T}\e^H\,\dd m=\infty\Longrightarrow\Pp(E_r=\T)=1.
\]
\end{theorem}

\begin{proof}
If $\sum_nv_n^{4/3}=\infty$, apply
\Cref{lem:high-moment} with $\varepsilon=1/3$.  If instead
$\sum_nv_n^{4/3}<\infty$, then
\[
 \int\e^H=\infty
 \overset{\eqref{eq:retained-poisson-cover}}{\Longrightarrow}
 \Pp\!\left(
 \limsup_{\substack{n\to\infty\\n\notin\mathcal R}}
 \widetilde C_n=\T\right)=1
 \overset{\eqref{eq:rank-inclusion}}{\Longrightarrow}
 \Pp(E_r^U=\T)=1.
\]
By \Cref{lem:limsup-cover-borel-tail}, the full-limsup-covering event is
Borel.  Since $(U_j)_{j\ge1}$ and $(X_j)_{j\ge1}$ have the same product law,
\[
 \Pp(E_r=\T)=\Pp(E_r^U=\T)=1.\qedhere
\]
\end{proof}

\section{Integrability implies noncovering}\label{sec:necessity}

The converse uses the uncovered volume rather than the multiscale population.
When $\e^H$ is integrable, the second moment of the uncovered volume is
uniformly comparable to the square of its mean.  A point remains uncovered
with positive probability, and the tail zero--one law then rules out full
limsup covering.

\begin{theorem}[Necessity]\label{thm:necessity}
\[
 \int_{\T}\e^H\,\dd m<\infty
 \Longrightarrow
 \Pp(E_r=\T)=0.
\]
\end{theorem}

\begin{proof}
Assume
\begin{equation}\label{eq:finite-energy}
 I_H:=\int_{\T}\e^{H(z)}\,\dd m(z)<\infty.
\end{equation}
Then, by Tonelli and \eqref{eq:int-u},
\begin{equation}\label{eq:square-sum}
 \sum_nv_n^2=\int_{\T}H\,\dd m\le I_H<\infty.
\end{equation}
Here $H\le\e^H$ because $H\ge0$.
Choose $s$ so that $v_n\le1/4$ for $n\ge s$.  Put
\begin{align}
 F_{s,N}&:=\T\setminus\bigcup_{n=s}^{s+N}B(X_n,r_n),\notag\\
 Y_{s,N}&:=m(F_{s,N}),\notag\\
 p_{s,N}&:=\prod_{n=s}^{s+N}(1-v_n),\notag\\
 K_{s,N}(z)&:=\prod_{n=s}^{s+N}
 \frac{1-2v_n+u_n(z)}{(1-v_n)^2}.
 \label{eq:finite-tail-objects}
\end{align}
Independence, Tonelli, and translation invariance give
\begin{equation}\label{eq:moments}
 \Ee Y_{s,N}=p_{s,N},\qquad
 \Ee Y_{s,N}^2=p_{s,N}^2\int_{\T}K_{s,N}(z)\,\dd m(z).
\end{equation}
Indeed, the joint probability that two points $x,y$ avoid the $n$th ball is
$1-2v_n+u_n(y-x)$; integrating first over $(x,y)$ and then using the
translation $z=y-x$ gives the second identity.
For $0\le u\le v\le1/4$,
\[
 0\le\frac{1-2v+u}{(1-v)^2}
 \le1+u+4v^2\le\e^{u+4v^2}.
\]
Hence, with
\[
 C_*:=\exp\!\left(4\sum_nv_n^2\right)I_H<\infty,
\]
\begin{equation}\label{eq:second-moment-bound}
 K_{s,N}(z)\le\exp\!\left(4\sum_nv_n^2\right)\e^{H(z)},\qquad
 \Ee Y_{s,N}^2\le C_*p_{s,N}^2.
\end{equation}
Therefore
\begin{equation}\label{eq:positive-residual}
 \Pp(Y_{s,N}>0)
 \ge\frac{(\Ee Y_{s,N})^2}{\Ee Y_{s,N}^2}
 \ge C_*^{-1}.
\end{equation}
Since $F_{s,N+1}\subseteq F_{s,N}$ are compact,
\[
 \Pp\!\left(\bigcap_{N\ge0}F_{s,N}\ne\varnothing\right)
 \ge C_*^{-1}>0.
\]
On this event a point is missed by every ball with index at least $s$, and
therefore cannot belong to $E_r$.  Hence $\Pp(E_r=\T)<1$.

By \Cref{lem:limsup-cover-borel-tail}, $\{E_r=\T\}$ is a tail event.
Kolmogorov's zero--one law and the strict upper bound just proved give
$\Pp(E_r=\T)=0$.
\end{proof}

\section{Proof of the main criterion}

\begin{proof}[Proof of \Cref{thm:main}]
\[
 \int_{\T}\e^H\,\dd m=\infty
 \overset{\text{\Cref{thm:sufficiency}}}{\Longrightarrow}
 \Pp(E_r=\T)=1,
\]
and
\[
 \int_{\T}\e^H\,\dd m<\infty
 \overset{\text{\Cref{thm:necessity}}}{\Longrightarrow}
 \Pp(E_r=\T)=0.
\]
\end{proof}

\section{Code availability and AI-use disclosure}

\subsection*{Code availability}

A Lean~4 formalization accompanying \Cref{thm:main} and its principal
supporting results is available at
\href{https://github.com/zyc111-234/shepp-formalization}
{github.com/zyc111-234/shepp-formalization}.  The repository contains the
paper-facing verification source, pinned Lean and Mathlib dependencies,
automated checks, and reproduction instructions.

\subsection*{AI-use disclosure}

AI-assisted tools, specifically OpenAI's ChatGPT and Codex, provided
substantive assistance with preliminary literature screening, proof
exploration, Lean~4 code generation and refactoring, consistency checks between
the formal development and the manuscript, selected stages of proof auditing,
manuscript organization, and language editing.  Their outputs were treated as
suggestions rather than as authoritative sources.  The author determined the
mathematical statements and scope, reviewed the arguments and their formal
counterparts, and assumes full responsibility for every mathematical claim,
proof, citation, and statement of priority in this paper.  No AI system is
credited as an author.

\appendix
\footnotesize

\section{Measurability of the transition maps}\label{app:measurability}

Let $E$ be compact metric, let $F$ be finite, and fix
$\rho:F\to(0,\infty)$.  Put
\[
 \mathcal K_0(E):=\mathcal K_+(E)\sqcup\{\varnothing\},
 \qquad
 \Cloud_F(E)=\prod_{a\in F}\bigl(\mathbb N_0\times E^{\mathbb N}\bigr),
\]
where $\mathcal K_+(E)$ has its Hausdorff Borel structure,
$\varnothing$ is isolated, and $\Cloud_F(E)$ has the product Borel
structure.  These are standard Borel spaces.  The required maps are
\begin{align}
 I:\mathcal K_0(E)^2&\longrightarrow\mathcal K_0(E),
 &(A,Q)&\longmapsto A\cap Q,\label{eq:appendix-intersection}\\
 D_\rho:\mathcal K_0(E)\times\Cloud_F(E)&\longrightarrow\mathcal K_0(E),
 &(A,\omega)&\longmapsto A\setminus\Cov_{E,\rho}(\omega).
 \label{eq:appendix-deletion}
\end{align}

The Hausdorff Borel structure on $\mathcal K_+(E)$ is generated by
\[
 \mathcal H_U:=\{K\in\mathcal K_+(E):K\cap U\ne\varnothing\},
 \qquad U\subseteq E\text{ open};
\]
see \cite[Chapter~1]{Molchanov2017}.  For $A\in\mathcal K_+(E)$, write
\[
 d_E(x,A):=\min_{a\in A}d_E(x,a).
\]
Fix nonempty open $U\subseteq E$ and define
\[
 K_m(U):=
 \begin{cases}
  \{x\in E:d_E(x,E\setminus U)\ge m^{-1}\},&U\ne E,\\
  E,&U=E,
 \end{cases}
 \qquad
 \mathcal M_U:=\{m\in\N:K_m(U)\ne\varnothing\}.
\]
Then $K_m(U)\uparrow U$, and for nonempty $A,Q$,
\[
 (A\cap Q)\cap U\ne\varnothing
 \Longleftrightarrow
 \exists m\in\mathcal M_U:\
 \min_{x\in K_m(U)}
 \bigl(d_E(x,A)+d_E(x,Q)\bigr)=0.
\]
Each minimum is continuous in $(A,Q)$.  The hit-set preimages and the
empty-output locus are therefore Borel, proving that
\eqref{eq:appendix-intersection} is Borel.

For $\mathbf q=(q_a)_{a\in F}\in\mathbb N_0^F$, let
\[
 \Cloud_F^{\mathbf q}(E)
 :=\{\omega\in\Cloud_F(E):(\omega_a)_1=q_a\text{ for every }a\in F\}.
\]
These Borel strata form a countable partition.  On
$\mathcal K_+(E)\times\Cloud_F^{\mathbf q}(E)$, write
$\omega_a=(q_a,\mathbf x_a)$ and, for $m\in\mathcal M_U$, define
\[
 \Phi_{m,\mathbf q}(A,\omega)
 :=
 \min_{y\in K_m(U)}
 \left[
  d_E(y,A)
  +\sum_{a\in F}\sum_{0\le i<q_a}
    \bigl(\rho(a)-d_E(y,\mathbf x_a(i))\bigr)_+
 \right].
\]
It depends on finitely many coordinate projections and is continuous on
each stratum.  Moreover,
\[
 D_\rho(A,\omega)\cap U\ne\varnothing
 \Longleftrightarrow
 \exists m\in\mathcal M_U:\ \Phi_{m,\mathbf q}(A,\omega)=0.
\]
Thus every hit-set preimage is Borel on each stratum and hence on their
countable union.  Taking $U=E$ handles the empty-output locus, so
\eqref{eq:appendix-deletion} is Borel.

For every standard Borel $S$, give
\[
 \Pop(S)=\bigsqcup_{q\in\mathbb N_0}S^{[q]}
\]
the coproduct Borel structure.  If
$S^\dagger=S\sqcup\{\dagger\}$ with $\dagger$ isolated, define
\[
 \operatorname{live}:\Pop(S^\dagger)\longrightarrow\Pop(S)
\]
by deleting the cemetery coordinates in increasing index order.  For fixed
$q$ and $J\subseteq[q]$, the stratum
\[
 \{(q,\mathbf x):
   \mathbf x_i\ne\dagger\Longleftrightarrow i\in J\}
\]
is Borel, and $\operatorname{live}$ restricts there to the finite coordinate
projection indexed by $J$.  Hence $\operatorname{live}$ is Borel.

Let $C$ be a finite totally ordered set and
\[
 T:S^\dagger\times W\longrightarrow(S'^\dagger)^C
\]
be Borel.  On the branch $\{q\}\times S^{[q]}$, the raw update
\[
 (q,\mathbf x,w)\longmapsto
 \left(
 q|C|,
 \bigl(T_c(\mathbf x_r,w)\bigr)_{(r,c)\in[q]\times C}
 \right)
\]
uses the parent-first lexicographic order on \([q]\times C\) and is a finite
product of Borel coordinate maps.  Composition with
$\operatorname{live}$ proves that the update in
\eqref{eq:shared-update} is Borel.

Apply this with $S=\mathcal S_{\tau(k,p-1)}$, $W=\Cloud_{\mathcal N_k}(\T)$,
and $D_\rho$ at deletion times; apply
\eqref{eq:appendix-intersection} at refinement times.  The maps in
\eqref{eq:deletion-transition}--\eqref{eq:refinement-transition}, the
populations $\Xi_j$, and their finite-population updates are Borel.  The
Bellman functions are Borel by the backward induction in
\Cref{lem:bellman-properties}.

\section{Probabilistic inputs}

Let $E$ be nonempty standard Borel.  Let
$\omega,\omega'\in\Cloud(E)$ be independent Poisson clouds with common
center law $\mu$ and respective count rates $\lambda,\lambda'$.  Then
\[
 \omega\oplus\omega'
 \stackrel{\mathrm d}{=}
 \bigl(P,(X_i)_{i\ge0}\bigr),
 \qquad
 P\sim\Pois(\lambda+\lambda'),
 \qquad
 (X_i)_{i\ge0}\stackrel{\mathrm{iid}}{\sim}\mu,
\]
with $P\perp\!\!\!\perp(X_i)_i$.  For finite $F$ and
$\omega=(q_a,\mathbf x_a)_{a\in F}\in\Cloud_F(E)$,
define
\[
 J_F(\omega)
 :=
 \sum_{a\in F}\sum_{i<q_a}\delta_{(\mathbf x_a(i),a)},
 \qquad
 \omega\le_F\omega'
 \Longleftrightarrow
 [\omega_a]\subseteq[\omega'_a]\quad(a\in F).
\]
Thus $\le_F$ is the active-point preorder.

\begin{lemma}[Cloud positive association]\label{lem:cloud-association}
Let $\mu$ be a probability measure and let the coordinates of $\Cloud_F(E)$ be
independent Poisson clouds with center law $\mu$ and rates
$(\lambda_a)_{a\in F}$.  Their product law satisfies
\[
 \Ee\prod_{i=1}^qG_i(\omega)
 \ge
 \prod_{i=1}^q\Ee G_i(\omega)
 \quad
 \substack{q\ge1;\\
 G_i\text{ bounded, nonnegative, Borel, increasing for }\le_F.}
\]
\end{lemma}

\begin{proof}
Let \(\mathbf N_f(E\times F)\) be the standard Borel space of finite
integer-valued measures, equipped with the counting-measure order.  The Borel
map \(J_F\) pushes the product cloud law to the finite-intensity Poisson law
with intensity
\[
 \Lambda(\dd x,\dd b)
 :=
 \sum_{a\in F}\lambda_a\,\mu(\dd x)\delta_a(\dd b).
\]
By measurable enumeration of finite point measures
\cite[Proposition~6.2]{LastPenrose2018}, with every atom repeated according to
its multiplicity, there is a Borel section
\[
 s:\mathbf N_f(E\times F)\longrightarrow\Cloud_F(E),
 \qquad
 J_F\circ s=\operatorname{id}.
\]
Every inactive tail is filled with a fixed point of \(E\).  For
\(\eta\in\mathbf N_f(E\times F)\), put
\[
 \operatorname{At}_a(\eta)
 :=
 \{x\in E:\eta(\{(x,a)\})>0\}.
\]
Then, for \(\omega\in\Cloud_F(E)\),
\[
 [s(\eta)_a]=\operatorname{At}_a(\eta),
 \qquad
 J_F(\omega)=\eta
 \Longrightarrow
 [\omega_a]=\operatorname{At}_a(\eta)
 \qquad(a\in F),
\]
and
\[
 \eta\le\eta'
 \Longrightarrow
 \operatorname{At}_a(\eta)\subseteq\operatorname{At}_a(\eta')
 \Longrightarrow
 s(\eta)\le_Fs(\eta').
\]
If \(G\) is increasing for \(\le_F\), then it is constant on every fiber of
\(J_F\).  Hence \(\overline G:=G\circ s\) is Borel and increasing on
\(\mathbf N_f(E\times F)\), and
\[
 G=\overline G\circ J_F.
\]
The finite-intensity form of Poisson positive association
\cite[Theorem~20.4]{LastPenrose2018}, followed by induction on $q$, proves the
assertion.
\end{proof}

Finally,
\[
 F_1\supseteq F_2\supseteq\cdots,
 \hspace{2em}F_i\ne\varnothing\text{ compact}
 \Longrightarrow
 \bigcap_iF_i\ne\varnothing.
\]
Tonelli, Chernoff bounds, Borel--Cantelli, Paley--Zygmund, and Kolmogorov's
zero--one law are invoked explicitly where used.

\end{document}